\documentclass[12pt, reqno]{amsart}

\usepackage{fourier}
\usepackage{nicefrac}

\usepackage[T1]{fontenc}
\usepackage[latin9]{inputenc}
\usepackage{amsthm}
\usepackage{amstext}
\usepackage{amssymb}

\usepackage{tikz}
\usetikzlibrary{positioning}

\usepackage[margin=1.5in]{geometry}

\usepackage{graphicx}
\usepackage{color}
\usepackage{url}

\usepackage{amsmath,amsfonts,latexsym,amsthm,amstext,enumerate,color}
\usepackage[numeric,initials,nobysame]{amsrefs}

\numberwithin{equation}{section}

\newtheorem{theorem}{Theorem}[section]
\newtheorem*{theorem*}{Theorem}
\newtheorem{lemma}[theorem]{Lemma}

\newtheorem*{observation*}{Observation}
\newtheorem{corollary}[theorem]{Corollary}

\theoremstyle{definition}{

}

\newcommand{\R}{\mathbb R}
\newcommand{\C}{\mathcal C}

\newcommand{\N}{\mathbb N}
\newcommand{\Z}{\mathbb Z}

\newcommand{\eps}{\varepsilon}

\newcommand{\A}{\mathcal A}
\renewcommand{\and}{\hbox{ {\rm and} }}

\newcommand{\E}{\mathbb{E}}
\renewcommand{\P}{\mathbf{P}}
\newcommand{\prob}{\P}

\def\Reff{R_{\rm eff}}
\def\EucB{B_{\rm euc}}
\def\deg{{\rm deg}}

\renewcommand{\epsilon}{\varepsilon}

\renewcommand{\dagger}{*}
\newcommand{\da}{\dagger}

\newcommand{\En}{{\mathcal{E}}}

\newcommand{\lra}{\leftrightarrow}
\newcommand{\be}{\begin{eqnarray}}
\newcommand{\ee}{\end{eqnarray}}

\date{}

\begin{document}
%\title{The uniform infinite planar triangulation is recurrent}
%\title{The UIPT is recurrent a.s.}
%\title{A uniformly drawn triangulation of the plane \\ is almost surely recurrent}
%\title{The uniform infinite planar triangulation \\ is almost surely recurrent}
\title{Recurrence of planar graph limits}

\author{Ori Gurel-Gurevich}
\address{Ori Gurel-Gurevich\hfill\break
Department of Mathematics \\
University of British Columbia.}
\email{origurel@math.ubc.ca}
\urladdr{}

\author{Asaf Nachmias}
\address{Asaf Nachmias\hfill\break
Department of Mathematics \\
University of British Columbia.}
\email{asafnach@math.ubc.ca}
\urladdr{}

\begin{abstract} We prove that any distributional limit of finite planar graphs in which the degree of the root has an exponential tail is almost surely recurrent. As a corollary, we obtain that the uniform infinite planar triangulation and quadrangulation (UIPT and UIPQ) are almost surely recurrent, resolving a conjecture of Angel, Benjamini and Schramm \cite{AS, BS}.

We also settle another related problem of \cite{BS}. We show that in any bounded degree, finite planar graph the probability that the simple random walk started at a uniform random vertex avoids its initial location for $T$ steps is at most ${C \over \log T}$.

% $M$ there exists $C$ such that on any planar graph of degree at most $M$ the probability that the simple random walk $(X_n)_{n \geq 0}$  $X_0$ for $n$ steps is at most ${C \over \log n}$.
%The proof involves the circle packing theorem together with discrete potential theory.
\end{abstract}

\maketitle

%\tableofcontents

%\section{Introduction}
\section{Introduction}
A {\em distributional limit} of finite graphs $G_n$ is a random rooted infinite graph $(U,\rho)$ with the property
that neighborhoods of $G_n$ around a random vertex converge in distribution to neighborhoods of $U$ around $\rho$, see precise definitions below. This limit was defined by Benjamini and Schramm \cite{BS}. Their motivation was the study of infinite random planar maps, a widely studied model in the probability, combinatorics and statistical physics communities for generic two-dimensional geometries and quantum gravity (see \cite{LeM, ADJ, DS, AS} and the references within). The canonical example of such a limit is Angel and Schramm's \cite{AS} {\em uniform infinite planar triangulation} (UIPT) and is obtained by taking the distributional limit of a uniform random triangulation on $n$ vertices. Here a triangulation is a simple planar graph in which every face has $3$ edges. %All the graphs in this paper have no loops or parallel edges.

% Other models of random geometries exist and our results hold for them as well.

 %One could take limits of quadrangulations (or any $p$-angulations for any $p$) but this is usually a matter of convenience and there is no qualitative difference between the different models, in other words, they are in the same universality class.

These authors conjectured that the UIPT is almost surely recurrent (see \cite[Conjecture 1.12]{AS} and \cite[Page 3]{BS}). It is shown in \cite{BS} that a distributional limit of uniformly bounded degree finite planar graphs is almost surely recurrent. However, the degrees of random planar maps and the UIPT are unbounded so one cannot appeal to this result. In this paper we prove this conjecture.

\begin{theorem}\label{mainthm} Let $(U,\rho)$ be a distributional limit of planar graphs such that the degree of $\rho$ has an exponential tail. Then $U$ is almost surely recurrent.
\end{theorem}

\begin{corollary} \label{maincor} The UIPT is almost surely recurrent.
\end{corollary}

The UIPT's recurrence conjecture has had strong circumstantial evidence supporting it. One such evidence is the recent result of Gill and Rohde \cite{RG} asserting that the natural Riemann surface associated with the UIPT (obtained by gluing equilateral triangles together according to the combinatorics of the graph) is almost surely parabolic, that is, Brownian motion on this surface is recurrent. Another such evidence, found by Benjamini and Curien \cite{BC2}, is that the UIPT is Liouville, that is, every bounded harmonic function on it is constant. Every recurrent graph is Liouville, and when $G$ is a bounded degree planar graph the Liouville property implies recurrence \cite{BS2} (the bounded degree condition in the last statement is necessary). \\

A popular variation of the UIPT is the {\em uniform infinite planar quadrangulation} (UIPQ) and is defined similarly with the role of triangulations replaced by quadrangulations. This model was constructed by Krikun \cite{Krik} (see also \cite{CMM}) and has received special attention since it is appealing to study it using bijections with random labeled trees.
%, we expand on this below.

\begin{corollary}\label{maincor2} The UIPQ is almost surely recurrent.
\end{corollary}

Our approach allows us to answer another problem posed by Benjamini and Schramm (Problem 1.3 in \cite{BS}). Let $G$ be a finite graph and consider the simple random walk on it $(X_t)_{t \geq 0}$ where $X_0$ is a uniform random vertex of $G$. Let $\phi(T,G)$ be the probability that $X_t \neq X_0$ for all $t=1,\ldots, T$. For any $D\geq 1$ define
$$ \phi_D(T) = \sup \big \{ \phi(T,G) : G \hbox{ is planar with degrees bounded by } D \big \} \, .$$
The almost sure recurrence of a distributional limit of planar graphs of bounded degree (the main result of \cite{BS}) is equivalent to $\phi_D(T) \to 0$ as $T\to \infty$ for any fixed $D$. It is asked in \cite{BS} what is the rate of decay of this function. The probability of avoiding the starting point for $T$ steps on $\Z^2$ is of order ${1 \over \log T}$ so this lower bounds $\phi_D(T)$ since $\Z^2$ is a distributional limit of finite planar graphs. Here we provide a matching upper bound.

\begin{theorem} \label{mainthm3} For any $D\geq 1$ there exists $C<\infty$ such that for any $T\geq 2$
$$ \phi_D (T) \leq {C \over \log T} \, .$$
\end{theorem}

Our results are related to two active research areas: graph limits and random planar maps. Let us briefly expand on each in order to introduce some definitions and background.
%set up the basic definitions used in this paper.

\subsection{Distributional graph limits}\label{sec-dlimit} The notion of the distributional limit of a sequence of graphs was introduced by Benjamini and Schramm \cite{BS}. With slightly different generality this was studied by Aldous and Steele \cite{AldSte} under the name ``local weak limit'' and by Aldous and Lyons \cite{AldLyo} under the name ``random weak limit''. This limiting procedure is best suited for graphs with bounded average degree and is hence natural in the setting of finite planar graphs. Convergence of sequences of dense graphs requires quite a different treatment (see \cite{LS, BCLSV}) though some interesting connections between the two are emerging (see \cite{BCKL}).

A rooted graph is a pair $(G,\rho)$ where $G$ is a graph and $\rho$ is a vertex of $G$. For any integer $r \geq 0$ we write $B_G(\rho,r)$ for the ball around $\rho$ of radius $r$ in the graph distance. The space of rooted graphs is endowed with a natural metric: the distance between $(G,\rho)$ and $(G', \rho')$ is ${1 \over \alpha+1}$ where $\alpha$ is the supremum over all $r$ such that $B_G(\rho,r)$ and $B_{G'}(\rho',r)$ are isomorphic as rooted graphs. Let $G_n$ be a sequence of finite graphs and let $\rho_n$ be a random vertex of $G_n$ drawn according to some probability measure on the vertices of $G_n$. We say that $(G_n, \rho_n)$ has distributional limit $(U,\rho)$, where $(U,\rho)$ is a random rooted graph, if for every fixed $r>0$ the random variable $B_{G_n}(\rho_n,r)$ converges in distribution to $B_U(\rho,r)$.

It makes sense to choose the random root according to the stationary distribution in $G_n$ (in which the probability of choosing a vertex is proportional to its degree) because then the resulting limit $(U,\rho)$ is invariant under the random walk measure, that is, $(U,\rho)$ has the same distribution as $(U,X_1)$ where $X_1$ is a uniform random neighbor of $\rho$. It is also common to choose the root according to the uniform distribution on the vertices of $G_n$. This may lead to a different distributional limit. However, in our setting this does not matter as we now explain. Let $\rho_n^\pi$ and $\rho_n^u$ be random roots of $G_n$ drawn according to the stationary and uniform distributions, respectively. If the average degree of $G_n$ is bounded by some number $D$ and $G_n$ has no isolated vertices, then it is immediate that for any event $A$ on rooted graphs we have $\prob((G_n,\rho_n^u)\in A) \leq D \prob((G_n,\rho_n^\pi)\in A)$. Hence, the distributional limit of $(G_n, \rho_n^u)$ is absolutely continuous with respect to the limit of $(G_n,\rho_n^\pi)$. In fact, an appeal to H\"older's inequality shows that if the degree distribution of $G_n$ has a bounded $(1+\eps)$-moment, then the two limits are mutually absolutely continuous with respect to each other (we do not use this fact in this paper though).

%When $\rho_n$ is chosen uniformly at random between the vertices of $G_n$ we say that this is an {\em unbiased} distributional limit.

\subsection{Random planar maps} Random planar maps is a widely studied topic at the intersection of probability, combinatorics and statistical physics. We give here a very brief account of this topic and refer the interested reader to \cite{AS, LeM} and the many references within. The enumerative study of planar maps was initiated by Tutte \cite{T} who counted the number of planar graphs of a given size of various classes, including triangulations. Cori and Vauquelin \cite{CV}, Schaeffer \cite{Sc} and Chassaing and Schaeffer \cite{CS} provided robust bijections between planar maps and labeled trees --- the specifics of these bijections change depending on the class of the planar maps considered and many variations and extensions are known. The common to all of these is that random planar maps can be constructed from random labeled trees. This approach has shed a new light on the asymptotic geometry of random maps and spurred a new line of research: limits of large planar random maps. There are two natural notions of limits of random planar maps: the scaling limit and the aforementioned distributional limit.

%
%this new approach have shed new light on the  and geometry of random planar maps and spurred great interest. There are two natural notions of limits of planar maps leading to two related yet independent lines of research: {\em scaling limits} and {\em distributional} or {\em local limits}.
%Most relevant to this work are two related yet  of random planar graphs.
%Recall that a {\em triangulation} is a planar graph that can be embedded in the plane such that all faces (including the outer face) are triangles. All graphs in this paper are simple, that is, they have no loops or double edges.
%We will not describe this area entirely and we refer the readers to the comprehensive surveys [ref ref]. The study was initiated by the work of Tutte [REF] who precisely enumerated using generating functions

In the study of scaling limits of random planar maps, one considers the random finite map $T_n$ on $n$ vertices as a random metric space induced by the graph distance, scales the distances properly (it turns out that $n^{-1/4}$ is the correct scaling) and studies its limit in the Gromov-Hausdorff sense. The existence of such a limit was first suggested by Chassaing and Schaefer \cite{CS}, Le Gall \cite{Le1}, Marckert and Mokkadem \cite{MM} who named it the {\em Brownian map}. The challenges involved in this line of research are substantial --- existence and uniqueness of the limit are the first step, but even more challenging is the issue of universality, that is, that random planar maps of different classes exhibit the same limit, up to parametrization. For the case of random $p$-angulations this research has recently culminated in the work of Le Gall \cite{Le2} who established this for $p=3$ and all even $p$ and independently Miermont \cite{Mier} for the case $p=4$. It remains open to prove this for all $p$.

%
%one needs to prove tightness, uniqueness and universality (that is, that the limit of random $p$-angulations is the same, up to parameterization, for any $p$). This research has recently culminated in the work

The study of distributional limits, while bearing some similarities, is independent of the scaling limit questions. Let $G_n$ be a random planar triangulation and $\rho_n$ a random vertex chosen uniformly (or according to the stationary measure, as mentioned above). Angel and Schramm \cite{AS} showed that a distributional limit exists and that it is a one-ended infinite triangulation of the plane almost surely. They termed this limit as the {\em uniform infinite planar triangulation} (UIPT). The uniform infinite planar quadrangulation (UIPQ) was later constructed by Krikun \cite{Krik}.
 %In general it is possible to construct limits for $p$-angulations for any fixed $p$ but we do not need this here since our result applies for any subsequential limit.

The research in this area is focused on almost sure geometric properties of this limiting geometry. It is an interesting geometry and the comparison of it with the usual Euclidean geometry is especially striking. It is invariant, planar and polynomially growing, yet very fractal: Angel \cite{A} showed that a ball of radius $r$ has volume $r^{4+o(1)}$ and the boundary component, separating this ball from infinity, has size $r^{2+o(1)}$ (see also \cite{CS}). This suggests that the random walk on the UIPT/UIPQ should be subdiffusive, that is, that the typical distance of the random walk from the origin after $t$ steps is $t^{\beta+o(1)}$ for some $\beta \in [0, 1/2)$. Benjamini and Curien \cite{BC1} show that $\beta \leq 1/3$ in the UIPQ, however, it is believed that the true exponent is $\beta=1/4$.

\subsection{Sharpness} Theorem \ref{mainthm} is sharp in the following sense. For any $\alpha\in(0,1)$ there exists a distributional limit of planar graphs $(U,\rho)$ such that $\prob( \deg(\rho) \geq k) \leq C e^{-c k^\alpha}$ for some $C,c$ that is transient almost surely. Indeed, let $T_h$ be a binary tree of height $h$ and replace each edge at height $k=1,\ldots,h$ from the leaves by a disjoint union of $k^{1/\alpha}$ paths of length $2$ (or parallel edges). In the distributional limit of $T_h$ as $h\to \infty$, almost surely, the effective resistance from the root to infinity is at most $2\sum_{k=1}^{\infty} k^{-1/\alpha} <\infty$ hence it is transient. Furthermore, the probability that the degree of a uniformly chosen vertex of $T_h$ is at least $k$ can easily be computed to be of order $e^{-ck^\alpha}$.

\section{Preliminaries on circle packing and electric networks}
\begin{figure}
\begin{center}
\includegraphics[width=0.5\textwidth]{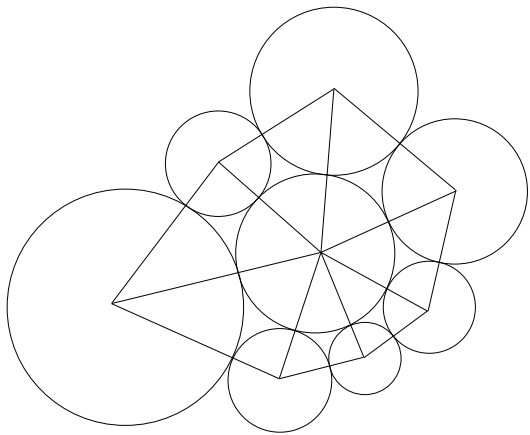}
\caption{A circle packing and its tangency graph}
\end{center}
\label{fig.circlepack}
\end{figure}

\subsection{Circle packing} Our proof relies, as in \cite{BS}, on the theory of circle packing, which we briefly describe below. We refer the reader to \cite{St} and \cite{R} for further information on this fascinating topic. A circle packing is a collection of circles in the plane with disjoint interiors. The tangency graph of a circle packing is a planar graph $G=(V,E)$ in which the vertex set $V$ is the set of circles and two circles are neighbors if they are tangent in the packing. The degree of a circle in the packing is its degree in the tangency graph. See Figure \ref{fig.circlepack}. The Koebe-Andreev-Thurston Circle Packing Theorem (see \cite{St}) asserts that for any finite planar graph $G=(V,E)$ there exists a circle packing in the plane which has tangency graph isomorphic to $G$. Furthermore, if $G$ is a triangulation, then this packing is unique up to M\"{o}bius transformations of the plane and reflections along lines. We will frequently use a simple but important fact known as the Ring Lemma \cite{RS}. If a circle $C$ is completely surrounded by $D$ other circles $C_0, \ldots, C_{D-1}$ (that is, $C_i$ is tangent to $C_{i+1 \,\,\rm{mod}\,\, D}$ and to $C$), then the ratio $r/r_i$ between the radius of $C$ and $C_i$ is bounded above by a constant depending only on $D$. Thus, in a circle packing of a bounded degree triangulation (every inner circle is completely surrounded) the ratio of radii of every two tangent circles is bounded above and below by a constant depending only on $D$, with the possible exception of the three boundary circles. The Circle Packing Theorem and the Ring Lemma are the only facts about circle packing that we will use in this paper.

\subsection{Electric networks}\label{sec-elec} We use some classical facts about electric networks and their connections to random walks, we refer the reader to \cite{LP} for further information. Let $G=(V,E)$ be a finite graph with non-negative edge weights $\{c_e\}_{e \in E}$. We call these weights conductances and their inverses, $R_e=c_e^{-1}$, are called resistances (by convention $0^{-1}=\infty$).
For any two vertices $a\neq z$ define the effective resistance $\Reff(a \lra z; \{R_e\})$ between $a$ and $z$ as the minimum energy $\En(\theta)=\sum_{e\in E} R_e [\theta(e)]^2$ of any unit flow $\theta$ from $a$ to $z$. The unit flow attaining this minimum is called the unit current flow. We often write $\Reff(a \lra z)$ when all the conductances are $1$.

Given two disjoint sets of vertices $A$ and $Z$, the effective resistance $\Reff(A \lra Z; \{R_e\})$ between $A$ and $Z$ is the effective resistance between the two corresponding vertices in the graph obtained from $G$ by contracting the sets $A$ and $Z$ into single vertices and retaining the same resistances on the remaining edges. For convenience, if either $A$ or $Z$ are empty sets, then $\Reff(A \lra Z; \{R_e\})=\infty$. Now we may define effective resistances on infinite graphs --- in this case we will only compute effective resistances between disjoint sets $A$ and $Z$ such that $V \setminus (A \cup Z)$ is finite. A typical example is the effective resistance between a chosen vertex $\rho$ and the complement of a finite set containing $\rho$. When $G$ is infinite we define the effective resistance from $a$ to $\infty$ as
$$ \Reff(a \lra \infty; \{R_e\}) = \lim_{n \to \infty} \Reff(a \lra G \setminus B_n; \{R_e\}) \, ,$$
where $\{B_n\}$ is any sequence of finite vertex sets which exhaust $G$ (the limit does not depend on the choice of exhausting sequence).

For a function $g:V \to \R$, the Dirichlet energy is defined as
$$ \En(g) = \sum_{e=(x,y) \in E} c_e \big [g(x) - g(y) \big ]^2 \, .$$
We will use the dual definition of effective resistance, that is, the discrete Dirichlet principle (see Exercise 2.13 of \cite{LP}) stating that
\be\label{dirichlet} { 1 \over \Reff(A \lra Z; \{R_e\}) } = \min \big \{ \En(g) \, : \, g:V\to \R , g_{|A}=0, g_{|Z}=1 \big \} \, .\ee

Consider the network random walk $(X_n)_{n \geq 0}$ on $G$ with transition probabilities $p(x,y) = c_{(x,y)} [\sum_{y : (x,y)\in E} c_{(x,y)}]^{-1}$ and write $\prob_x$ for the probability measure of a network random walk started at $X_0=x$. Write $\tau$ for the stopping time $\tau = \min\{ n \geq 1 : X_n \in \{a,z\}\}$. It is classical (stemming from the fact that the minimizer of (\ref{dirichlet}) is the unique harmonic function with the corresponding boundary values, see \cite{LP}) that
\be\label{rwresist} \Reff(a \lra z; \{R_e\}) = {1 \over \prob_a(X_{\tau}=z) \displaystyle \sum_{y : (a,y)\in E} c_{(a,y)}} \, .\ee
This gives a useful electrical interpretation of recurrence. An infinite network $(G; \{R_e\})$ is recurrent if and only if $\Reff(a \lra \infty; \{R_e\})=\infty$.
%if there exist finite sets $\{B_n\}_{n \geq 0}$ that exhaust the graph and a vertex $a$ such that
%$$\lim_{n\to \infty} \Reff(a \lra G \setminus B_n; \{R_e\}) = \infty \, .$$
It is not too hard to see that this implies the following two useful criteria for recurrence/transience. First, an infinite graph is $G$ is recurrent if and only if for some vertex $a$ there exists $c>0$ such that for any integer $m \geq 0$ there exists a finite vertex set $B$ such that
\be\label{recurrencecriterion}  \Reff( B_G(a,m) \lra G \setminus B; \{R_e\}) \geq c \, ,\ee
see \cite[Lemma 9.22]{LP}. Secondly, a network is transient if and only if there exists a unit flow from some vertex $a$ to $\infty$ with finite energy.

%
%there exists finite sets $\{G_n\}_{n\geq 0}$ that exhaust the graph, a vertex $a$ and a constant $c>0$ such that

Another classical connection between random walks and effective resistances is known as the commute time identity \cite{LP} stating that
\be\label{commute} \E _a \tau_z + \E _z \tau _a = 2 \Reff(a \lra z) \sum_{e \in E} c_e  \, ,\ee
where $\tau_v$ is the hitting time of $v$ and $\E_x$ is the corresponding expectation operator of $\prob_x$. We will also use the following bound, which is an immediate consequence of (\ref{rwresist}). Given a finite network and three vertices $x,y,z$ we have
\be\label{exercise} {1 \over \Reff(x \lra \{y,z\})} \leq {1 \over \Reff(x \lra y)} + {1 \over \Reff(x \lra z)} \, .\ee

%It is well known that $\Reff(x \lra y)$ satisfies the triangle inequality.
Finally, we will use the following easy bound.
\begin{lemma} \label{resistintriangle} Let $G=(V,E)$ be a finite network with resistances $\{R_e\}$ and two vertices $a$ and $z$. Let $A \subset V$ such that $a \in A$ and $z\not \in A$ and define $R^A_e = R_e$ for each edge $e$ that has both endpoints in $A$ and $R^A_e=\infty$ otherwise. Then
$$ \Reff(a \lra z ; \{R_e\}) \leq \Reff(A \lra z ; \{R_e\}) + \max_{v \in A} \Reff(a \lra v; \{R^A_e\}) \, .$$
\end{lemma}
\begin{proof} Assume without loss of generality that $\max_{v \in A} \Reff(a \lra v; \{R^A_e\}) < \infty$. Consider the unit current flow $\theta^A$ from $A$ to $z$ in $(G,R_e)$ and for each $v\in A\setminus \{a\}$ let $\theta^{v,A}$ be the unit current flow in $(G, \{R_e^A\})$ from $a$ to $v$. For each $v \in A$ write
$$ \alpha_v = \sum_{u: u \not \in A, u \sim v} \theta^A(v,u) \, .$$
Since $\theta^A$ is the unit current flow we have that $\sum_{v\in A} \alpha_v = 1$ and $\alpha_v \geq 0$ for all $v\in A$. We define a new flow $\theta$ from $a$ to $z$ in $(G,R_e)$ by setting $\theta(e) = \theta^A(e)$ for any edge $e$ that has at least one endpoint not in $A$, and if $e$ has both endpoints in $A$ we set $\theta(e) = \sum_{v \in A\setminus\{a\}} \alpha_v \theta^{v,A}(e)$. This defines a unit flow from $a$ to $z$. The contribution to the energy of $\theta$ coming from edges having at least one endpoint not in $A$ is the energy of $\theta^A$ which equals $\Reff(A \lra z ;  \{R_e\})$ and the contribution coming from edges with two endpoints in $A$ is at most
\begin{eqnarray*} \sum_{e} R_e \big [ \sum_{v\in A\setminus\{a\}} \alpha_v \theta^{v,A}(e) \big ]^2 &\leq&  \sum_{v \in A \setminus \{a\}} \alpha_v \sum_{e} R_e [\theta^{v,A}(e)]^2 = \sum_{v \in A \setminus \{a\}} \alpha_v \Reff(a \lra v; \{R^A_e\}) \\
&\leq& \max_{v \in A} \Reff(a \lra v;  \{R^A_e\}) \, ,\end{eqnarray*}
where the first inequality is by Jensen's inequality.
\end{proof}

%
% the flow minimizing the energy is $\theta(x,y)=h(x)-h(y)$ where $h$ is the voltage function, that is, the unique function having $h(a)=0$ and $h(z)=1$ and is harmonic at every vertex of $G \setminus \{a,z\}$ (i.e., $h(x) = \sum_{y : (x,y)\in E} c_e h(y)$). This gives a useful connection to random walks.  Then it is easy to see that $h(x)$ equals $\prob_x(X_n \hbox{ visits } z \hbox{ before } a)$.

\section{Distributional limits of bounded degree graphs and circle packing}
Let $G_n$ be a sequence of finite planar graphs of bounded degree and assume that it has distributional limit $(U,\rho)$. The main result of \cite{BS} is that $U$ is almost surely recurrent. Our goal in this section is the prove the following theorem, providing a quantitative bound on the growth of the resistance.

\begin{theorem} \label{quantrecurrence} Let $(U,\rho)$ be the distributional limit of finite planar graphs of bounded degree. Then $(U,\rho)$ almost surely satisfies the following. There exists $c>0$ such that for any $k \geq 0$ there exists a finite set $B_k \subset U$ with $|B_k| \leq c^{-1} k$ and
$$ \Reff ( \rho \lra U \setminus B_k ) \geq c \log k \, .$$
\end{theorem}

We begin with some basic estimates relating circle packing and resistances.

\subsection{Circle packing and resistance} %Next we prove some useful resistance bounds valid for a deterministic circle packing.

Given a circle packing $P = \{C_v : v \in G\}$ of a graph $G=(V,E)$ and given a domain $D \subset \R^2$ we write $V_D \subset V$ for the set of vertices such that their corresponding circles have centers in $D$. We also write $\EucB(p,r)$ for the Euclidean ball of radius $r$ around $p$.

\begin{lemma} \label{resistannulus} Let $P = \{C_v : v \in G\}$ be a circle packing of a finite graph $G=(V,E)$ such that the ratio of radii of two tangent circles is bounded by $K$. Then for any $\alpha>1$ there exists $c=c(K,\alpha)>0$ such that for all $r>0$ and all $p \in \R^2$
$$ \Reff \big ( V_{\EucB(p,r)} \lra V_{\R^2 \setminus \EucB(p, \alpha r)} \big ) \geq c \, ,$$
provided that both sets $V_{\EucB(p,r)}$ and $V_{\R^2 \setminus \EucB(p, \alpha r)}$ are nonempty.
\end{lemma}
\begin{proof}
%By our assumption on $C_\rho$ no circle with center in $\EucB(p,r)$ has radius larger than $r$, and so by the bounded ratio %assumption we can choose $C>1$ so that there is no edge between $V_{\EucB(p,r)}$ and $V_{\R^2 \setminus \EucB(p, Cr)}$. Choose %such $C$ and we now show the lower bound on the resistance.
In the case where $|V_{\EucB(p,r)}|=1$ the resistance is at least the inverse of the degree of the vertex. Since the ratio between the radii of tangent circles is bounded, the degrees of $G$ are bounded by some $D=D(K)<\infty$ so the resistance is at least $D^{-1}$ in this case. Thus, let us assume that $|V_{\EucB(p,r)}|>1$.
Define a function $f:\R^2 \to \R$ by
$$ f(x) = \begin{cases} 0 &\mbox{if } x \in \EucB(p,r) \\
{||x-p||-r \over (\alpha-1)r} & \mbox{if } x \in \EucB(p,\alpha r) \setminus \EucB(p,r) \\
1 & \mbox{if } x \in \R^2 \setminus \EucB(p,\alpha r) .\end{cases}
$$
Note that $f$ is $((\alpha-1)r)^{-1}$-Lipschitz. We define $g:V \to \R$ by setting $g(v)=f(\rho_v)$, where $\rho_v$ is the center of $C_v$ and bound its Dirichlet energy. For every edge $(u,v)$
$$ || g(u) - g(v)|| \leq {|| \rho_u - \rho_v || \over (\alpha-1)r } = { (r_u + r_v) \over (\alpha-1)r} \leq {(K+1) r_u \over (\alpha-1)r} \, .$$
Edges for which both $\rho_u$ and $\rho_v$ are not in $\EucB(p,\alpha r)$ contribute $0$ to the energy. Since $|V_{\EucB(p,r)}|>1$, for every edge $(u,v)$ that has one of $\rho_u$ or $\rho_v$ in $\EucB(p,\alpha r)$ the circles $C_u$ and $C_v$ are both contained in $\EucB(p, K_2 r)$ for some $K_2(K, \alpha)< \infty$. Since the interiors of the circles $\{C_v\}_{v\in V}$ are disjoint the contribution to the energy is at most
$$ \sum_{(u,v)\in E} ||g(u) - g(v)||^2 \leq {D \cdot (K+1)^2 \cdot \hbox{Area}\big [\EucB(p,K_2 r) \big ] \over  ((\alpha-1)r)^{2} } \leq K_3 \, ,$$
where $K_3=K_3(K,\alpha)<\infty$, concluding our proof by \eqref{dirichlet}.
\end{proof}

\begin{corollary} \label{resisttolarge} Let $P$ be a finite circle packing in $\R^2$ such that the ratio of radii of two tangent circles is bounded by $K$ and such that there exists a circle in $P$ entirely contained in
$\EucB(0,1)$. Then there exists a constant $c=c(K)>0$ such that for all radii $r\geq 2$ we have
$$ \Reff \big ( V_{\EucB(0,1)} \lra V_{\R^2 \setminus \EucB(0, r)} \big) \geq c \log r \, ,$$
provided that $V_{\R^2 \setminus \EucB(0, r)}$ is nonempty.
%as long as both sets $V_{\EucB(0,r_1)}$ and $V_{\R^2 \setminus \EucB(0, r_2)}$ are nonempty.
\end{corollary}
\begin{proof} Since there is a circle entirely contained in $\EucB(0,1)$ and using the bounded ratio assumption we get that there exists $C=C(K)<\infty$ such that for all $r'\geq 1$ there is no edge between $V_{\EucB(0,r')}$ and $V_{\R^2 \setminus \EucB(0,Cr')}$. Assume that $r\geq C$ and consider the $k$ disjoint annuli $A_1, \ldots, A_k$ where $k= \lfloor \log_{C} r \rfloor$ and $A_j= \EucB(0,C^{j}) \setminus \EucB(0,C^{j-1})$ so that $A_j$ is contained in $\EucB(0,r) \setminus \EucB(0,1)$ for all $j$. There are no edges in $G$ between $V_{A_j}$ and $V_{A_{\ell}}$ for $1 \leq j \leq \ell - 2 \leq k-2$ hence each $V_{A_j}$ is a cut-set separating $V_{A_{j-1}}$ from $V_{A_{j+1}}$. By Lemma \ref{resistannulus} we have $\Reff(V_{A_j} \lra V_{A_{j+2}}) \geq c$ for some $c=c(K)>0$. Summing these resistances using the series law (see \cite{LP}) yields that $\Reff(V_{A_0} \lra V_{A_k}) \geq ck/2$. Finally, if $2 \leq r \leq C$, then the resistance is bounded below by another constant using Lemma \ref{resistannulus}.
\end{proof}
%
%\begin{corollary} \label{resisttosmall}  Let $P$ be a connected finite circle packing in $\R^2$ such that the ratio of radii of two adjacent circles is bounded by $K$ and such that there is a circle $C_\rho$ centered at the origin with radius $1$. Let $r \geq 1$ and $\gamma>0$ be numbers and $p\in \R^2$ be a point such that $||p|| \geq r+2\gamma$ and such that there is a circle of $P$ entirely contained in $\EucB(p,\gamma)$. Then
%$$ \Reff(V_{\EucB(0,r)} \lra V_{\EucB(p,\gamma)}) \geq c \log \Big ( {||p||-r \over \gamma }\Big ) \, ,$$
%for some $c=c(K)>0$.
%\end{corollary}
%\begin{proof} Surround $p$ be the $k$ annuli where $k= \lfloor \log_C (r_2 ||p||) \rfloor$ such that the $j$-th annulus has inner radius $C^j r_2^{-1}$ and outer radius $C^{j+1} r_2^{-1}$, where $C$ is chosen as in the previous corollary. Repeat the analysis of the previous corollary to complete the proof.
%\end{proof}

\subsection{Finite planar graphs} We recall the key lemma of \cite{BS}. Let $\C \subset \R^2$ be a finite set of points. For each $w \in \C$ we write $\rho_w$ for its {\em isolation radius}, that is, $\rho_w = \inf\{ |v-w| : v \in \C \setminus \{w\}\}$. Given $\delta>0, s>0$ and $w\in \C$ we say that $w$ is $(\delta,s)$-{\em supported} if in the disk of radius $\delta^{-1} \rho_w$ around $w$ there are at least than $s$ points of $\C$ outside of every disk of radius $\delta \rho_w$, in other words, if
$$ \inf _{p \in \R^2} \big | \C \cap B(w,\delta^{-1} \rho_w) \setminus B(p,\delta \rho_w) \big | \geq s \, .$$
Benjamini and Schramm \cite[Lemma 2.3]{BS} prove that for all $\delta\in (0,1)$ there is a constant $c = c(\delta)$ such that for
every finite $\C \subset \R^2$ and every $s \geq 2$ the set of $(\delta, s)$-supported points in $\C$
has cardinality at most $c|\C|/s$. In the following we bound $c(\delta)$.

\begin{lemma}\label{BSquant}  There exists a universal constant $A>0$ such that for all $\delta\in(0,1/2)$ and $s\geq 2$ and any finite set $\C \subset \R^2$ the number of $(\delta,s)$-supported points in $\C$ is at most ${A |\C| \delta^{-2} \log (\delta^{-1}) \over s}$.
\end{lemma}
\begin{proof} To understand the proof of this one must first read \cite[Lemma 2.3]{BS}. Our lemma is a straightforward calculation of the constants appearing in the last paragraph of the proof of \cite[Lemma 2.3]{BS}.
Indeed, in the notation of \cite{BS}, $c=c(\delta)$ is $c=2c_0^{-1} c_1$. We estimate these constants below. Given $\delta$, $c_1(\delta)$ is the number of cities in any square $S$ and is at most $A_0 \delta^{-2}$ for some universal $A_0<\infty$. The probability $c_0(\delta)$ that there exists a square that has edge length in the range $[4\delta^{-1} \rho_w, 5\delta^{-1} \rho_w]$ is readily seen to be at least $A_1 \log ^{-1} (\delta)$ for some universal $A_1 > 0$.
\end{proof}

\begin{corollary} \label{finitegraph} Let $G$ be a finite planar triangulation and $P=\{C_v : v\in G\}$ be an arbitrary circle packing of $G$. Let $\rho$ be a random uniform vertex of $G$ and let $\widehat{P}=\{\widehat{C}_v : v\in G\}$ be the circle packing obtained from $P$ by translating and dilating so that $\widehat{C}_\rho$ has radius $1$ and is centered around the origin. Then there exists a universal constant $A>0$ such that for any $r\geq 2$ and any $s\geq 2$
$$ \prob \big ( \forall p\in \R^2 \,\,\, |V_{\EucB(0,r) \setminus \EucB(p,   r^{-1})}| \geq s \big ) \leq {A r^2 \log r \over s} \, .$$
\end{corollary}
\begin{proof} Apply Lemma \ref{BSquant} with $\C$ being the set of centers of $P$ and $\delta=r^{-1}$ and $s$. We deduce that the number of centers of $P$ that are $(r^{-1}, s)$-supported is at most $A s^{-1} |G| r^{2} \log r$. Since $\widehat{P}$ is a triangulation any circle $C_w$ with radius $r_w$ not in the boundary (the boundary has $3$ circles, which contributes a negligible $3|G|^{-1}$ to the probability) has $r_w \leq \rho_w \leq Cr_w$ for some universal constant $C>0$, concluding the proof.
%. Hence the probability of the required event is at most $A c^{-1} r^{-1} \log r$.
%Hence the event $\A_r$ that in $\widehat{P}^n$ there are at least $c r^3$ points in $\EucB(0,r) \setminus \EucB(p,Cr^{-1})$ for any $p\in \R^2$ has probability at most
\end{proof}

\begin{lemma} \label{finitegraphfinal} Let $G=(V,E)$ be a finite planar graph with degrees at most $D$ and let $\rho$ be a random uniform vertex. Then there exists $c=c(D)>0$ such that for all $k \geq 1$
$$ \prob \Big ( \exists B \subset V \hbox{ with } |B| \leq c^{-1} k \and \Reff(\rho \lra V\setminus B) \geq c \log k \Big ) \geq 1- c^{-1} k^{-1/3} \log k \, ,$$
where we interpret $\Reff(\rho \lra V \setminus B)=\infty$ when $B=V$.
\end{lemma}
\begin{proof} Without loss of generality it is enough to prove this for $k$ large enough. Assume first that $G$ is a triangulation and consider the circle packing $\widehat{P}$ from Corollary \ref{finitegraph}. Apply this Corollary with $r=k^{1/3}$ and $s=k$. We get that with probability at least $1-A k^{-1/3} \log k$ there exists $p\in \R^2$ such that $|V_{\EucB(0,r) \setminus \EucB(p, r^{-1})}| \leq k$. We proceed by analyzing two cases. If $|V_{\EucB(p,r^{-1})}| \leq 1$, then we set $B=V_{\EucB(0,r)}$ so that $|B| \leq k+1$.
In this scenario, if $V_{\R^2 \setminus \EucB(0,r)} = \emptyset$, then $B=V$ and the assertion holds trivially. If $V_{\R^2 \setminus \EucB(0,r)} \neq \emptyset$, then by Corollary \ref{resisttolarge} we have
$$ \Reff (\rho \lra U \setminus B) \geq c \log k \, ,$$
where $c=c(D)>0$.

In the case where $|V_{\EucB(p,r^{-1})}| \geq 2$ we take $B=V_{\EucB(0,r) \setminus \EucB(p,   r^{-1})}$. Since $G$ is a triangulation and $C_\rho$ is a circle at the origin of radius $1$, by the Ring Lemma we deduce that there exists some $c'=c'(D)>0$ such that the center of any circle other than $C_\rho$ is of distance at least $1+c'$ from the origin. Hence, when $r$ is large enough we must have that $||p|| \geq 1 + c'/2$. Clearly, one of the circles with centers in $\EucB(p,r^{-1})$ must have radius at most $r^{-1}$. Hence, $\EucB(p, 2r^{-1})$ entirely contains a circle and so we may scale and dilate so that Corollary \ref{resisttolarge} gives
$$ \Reff(V_{\EucB(p,2r^{-1})}\lra V_{\R^2 \setminus \EucB(p,c'/2)}) \geq c \log r \, ,$$
therefore
\be\label{quant.resisttosmall} \Reff(\rho \lra V_{\EucB(p,2r^{-1})}) \geq c \log r \, .\ee
Also, by Corollary \ref{resisttolarge} we have
\be\label{quant.resisttolarge} \Reff(\rho \lra V_{\R^2 \setminus \EucB(0,r)} ) \geq c \log r \, .\ee
The inequalities (\ref{quant.resisttosmall}) and (\ref{quant.resisttolarge}) together with (\ref{exercise}) concludes the proof when $G$ is a triangulation.

If $G$ is not a triangulation, then we add edges and vertices to extend $G$ into a finite planar triangulation $T$ in the zigzag fashion as in \cite[Proof of Theorem 1.1]{BS}. After this procedure the maximal degree and the number of vertices have multiplied by at most a universal constant $K$. Let $\rho_T$ be a uniform random vertex of $T$, by the proof in the case of triangulations, with probability at least $1-Ak^{-1/3} \log k$ there exists $B' \subset T$ with $|B'|\leq c^{-1} k$ and $\Reff(\rho_T \lra U\setminus B') \geq c \log k$. We take $B = B' \cap V(G)$. Obviously $|B| \leq c^{-1} k$ and by Rayleigh's monotonicity the effective resistance only grew. Lastly, $\prob(\rho_T \in V(G)) \geq 1/K$ so by incorporating $K$ into the constant $A$ we conclude the proof.
\end{proof}

%\subsection{Distributional limits}

\noindent {\bf Proof of Theorem \ref{mainthm3}.} Let $G=(V,E)$ be a finite planar graph with degree bounded by $D$ and let $T\geq 2$ be an arbitrary integer. Assume without loss of generality that $G$ is connected. Let $X_0$ be a uniform random vertex. Apply Lemma \ref{finitegraphfinal} with $k=T^{1/3}$ so that with probability at least $1-c^{-1}T^{-1/9}\log T$ there exists $B \subset V$ with $|B| \leq c^{-1} T^{1/3}$ and
$$ \Reff( X_0 \lra V \setminus B) \geq c \log T \, .$$
If this event occurred, by the commute time identity (\ref{commute}) and the fact that $\Reff(a \lra z)$ is at most the graph distance between $a$ and $z$, we have that
$$ \E_{X_0} \tau_{V \setminus B} \leq 2 D|B| \Reff(X_0 \lra V\setminus B) \leq 2D c^{-2} T^{2/3} \, ,$$
where $\tau_{V \setminus B}$ is the hitting time of the random walk at $V \setminus B$. If $\tau_{X_0} \geq T$, then either $\tau_{V\setminus B} \geq T$ or $\tau_{X_0} > \tau_{V \setminus B}$. Hence by Markov's inequality and (\ref{rwresist})
$$ \prob_{X_0} \big ( \tau_{X_0} \geq T \big ) \leq {2 D c^{-2} T^{2/3} \over T} + {1 \over c D \log T} \, .$$
Putting all these together gives that
$$ \varphi_D(T) \leq c^{-1}T^{-1/9}\log T + {c^{-2} D T^{2/3} \over T} + {1 \over c D \log T} \leq {C \over \log T} \, ,$$
for some $C=C(D)>0$.
\qed \\

\noindent{\bf Proof of Theorem \ref{quantrecurrence}.} For any $k=1,2,\ldots$ write $\A_k$ for the complement of the event
$$ \big \{ \exists B \subset U \hbox{ with } |B| \leq c^{-1} k \and \Reff(\rho \lra U\setminus B) \geq c \log k \big \} \, .$$
Lemma \ref{finitegraphfinal} gives that $\prob(\A_k) \leq c^{-1} k^{-1/3} \log k$ since $(U,\rho)$ is a distributional limit of finite planar graphs of bounded degree. Borel-Cantelli implies that $\A_{2^j}$ occurs for only finitely many values of $j$. If $\A_{2^j}$ does not occur, then for all $2^{j-1} \leq k \leq 2^j$ there exists $B \subset U$ with $|B| \leq 2c^{-1} k$ and $\Reff(\rho \lra U\setminus B) \geq c \log k$. \qed

\section{Reducing to bounded degrees}

\subsection{Bounded degree distributional limits with markings}

\begin{lemma} \label{rwhitmiss} Let $G=(V,E)$ be a finite network with two distinguished vertices $a \ne z$. Let $\{R_e\}$ and $\{R'_e\}$ be two sets of resistances on $E(G)$ and let $S \subset V \setminus \{a,z\}$ be such that $R_e=R'_e$ for any $e \not \in S\times S$. Then
$$ \big | \prob_a ( \tau_z < \tau_a ) - \prob'_a(\tau_z < \tau_a) \big | \leq \prob_a(\tau_S < \tau_{\{a,z\}}) \, ,$$
where $\prob$ and $\prob'$ are the network random walks with resistances $R$ and $R'$, respectively.
%and resistances $\{R_e\}_{e \in E}$. For any subset of vertices $S$ write $\tau_S = \inf \{ t: X_t \in S\}$ where $X_t$ is the network random walk. Let $S \subset V \setminus \{a,z\}$. Then for any other resistances $\{R'_e\}$ defined on $G$ such that $R'_e = R_e$ for any $e \not \in S \times S$ let $\prob '$ be the network random walk associated with $\{R'_e\}$, then
%$$ \big | \prob_a ( \tau_z < \tau_a ) - \prob'_a(\tau_z < \tau_a) \big | \leq \prob_a(\tau_S < \tau_{\{a,z\}}) \, .$$
\end{lemma}
\begin{proof} Immediate by coupling the two random walks until they hit $S\cup\{a,z\}$.
\end{proof}

Next, we consider a triplet $(U, \rho, M)$, where $(U, \rho)$ is a random rooted graph as before and $M$ is a marking function $M:E(U) \to \R^+$. Conditioned on $(U,\rho,M)$ consider the simple random walk $(X_n)_{n \geq 0}$ where $X_0=\rho$. We say that $(U, \rho, M)$ is {\em stationary} if $(U,\rho,M)$ has the same distribution as $(U, X_1, M)$ in the space of isomorphism classes of rooted graphs with markings (this concept is described with more details in \cite{AldLyo}). Given a marking $M$ we extend it to $M:E(U) \cup V(U) \to \R$ by putting $M(v) = \max_{e: v \in e} M(e)$ for $v \in V(U)$. We say that $(U,\rho,M)$ has an exponential tail with exponent $\beta>0$ if $\prob(M(\rho) \geq s) \leq 2e^{-\beta s}$ for all $s \geq 0$.

\begin{lemma} \label{highmarksdontmatter} Let $(U,\rho,M)$ be stationary, bounded degree random rooted graph with markings that has an exponential tail with exponent $\beta>0$. Then almost surely there exists $K<\infty$ such that for any finite subset $B \subset V(U)$ containing $\rho$, of size at least $K$, we have
$$ \bigg | \prob_\rho( \tau_{U \setminus B} < \tau_\rho) - \prob'_\rho(\tau_{U \setminus B} < \tau_\rho) \bigg | \leq {1 \over |B|} \, ,$$
where $\prob$ and $\prob'$ are the network random walk with resistances $R$ and $R'$, respectively, where $R_e=1$ for all $e\in E(U)$ and $R'$ are any resistances satisfying $R'_e=1$ whenever $M(e) \leq 30\beta^{-1} \log |B|$.
%then almost surely the following occurs. For any finite subset $B \subset U$ with $\rho\in B$ such that $|B|$ is large enough consider the finite network $G=(V,E)$ obtained from $U$ by contracting $U \setminus B$ to a single vertex $z$. Associate unit resistances with $G$ and let $G'$ be a network with the same underlying graph and resistances $\{R'_e\}$ satisfying $R'_e=1$ whenever $M(e) \leq 30\beta^{-1} \log |B|$. Then
%$$ \big | \prob_\rho( \tau_{U \setminus B} < \tau_\rho) - \prob'_\rho(\tau_{U \setminus B} < \tau_\rho) \big | \leq {1 \over |B|} \, ,$$
%where $P$ and $P'$ are the network random walk measures on $G$ and $G'$.
%
%
% let $s>0$ and $B \subset U$ with $|B|<\infty$ and $\rho \in B$. Consider the finite network $G=(V,E)$ obtained from $U$ by contracting $U \setminus B$ to the sink $z$ with unit resistances and $G'$ is the network with the same underlying graph and resistances $\{R'_e\}$ given as follows: for each edge $e$ such that $M(e) \geq s$ we put $R'_e = {1 \over M(e)}$ and otherwise put $R'_e = 1$. Then
\end{lemma}
\begin{proof} For any two integers $T,s\geq 1$ let $\A_{T,s}$ denote the event
$$ \A_{T,s} = \Big \{ \prob_\rho \big ( \exists t \leq T : M(X_t) \geq s \big ) \leq  T^3 e^{-\beta s/2} \Big \} \, ,$$
and note that this event is measurable with respect to $(U,\rho,M)$. Stationarity together with exponential tail implies that for any integer $t \geq 0$
$$ \E_{(U,\rho,M)} \big [ \prob_\rho (M(X_t) \geq s) \big ] \leq 2 e^{-\beta s} \, ,$$
hence the union bound gives
$$ \E_{(U,\rho,M)} \big [ \prob_\rho (\exists t \leq T : M(X_t) \geq s) \big ] \leq 2 T e^{-\beta s} \, .$$
By Markov's inequality
$$ \prob(\A_{T,s}^c) \leq {2e^{- \beta s /2} \over T^2} \, .$$
Borel-Cantelli implies that almost surely $A_{T, s}$ occurs for all but finitely many values of $T \in \N$ and $s\in \N$. For any finite set $B \subset U$ that contains $\rho$, by the commute time identity (\ref{commute}), the fact that $\Reff(\rho \lra U\setminus B)$ is at most the graph distance between $\rho$ and $U \setminus B$ and Markov's inequality
$$ \prob_\rho ( \tau_{U\setminus B} \geq T ) \leq { 2 D|B|^2 \over T} \, ,$$
where $D$ is the degree bound. Write $S= \{ v : v \in V(U) \and M(v) \geq s\}$, then for any $T,s$ such that $\A_{T,s}$ occurs we have
$$ \prob_\rho ( \tau_S < \tau_{\{\rho\}\cup (U\setminus B)} ) \leq { 2 D |B|^2 \over T} + T^3 e^{-\beta s/2} \, .$$
Now, take $T = 4D|B|^3$ and $s=30 \beta^{-1} \log |B|$ so that the right hand side is at most $|B|^{-1}$ when $|B|$ is large enough and apply Lemma \ref{rwhitmiss}.
%
%
%Let $S \subset V(G)$ be the set of vertices that has an incident edge $e$ with $M(e)\geq s$. By the union bound and stationarity, for any $T$ we have
%$$ \prob_\rho( \exists t \leq T \hbox{ with } X_t \in S ) \leq T \prob_\rho(\rho \in S) \, .$$
%Also, by the commute time identity [refer to book] and Markov's inequality
%$$ \prob_\rho ( \tau_z \geq T ) \leq {2 |E|^2 \over T} \, ,$$
%hence
%$$ \prob_\rho (  \tau_S < \tau _{\{a,z\}} ) \leq T \prob(\rho \in S) + {2 |E|^2 \over T} \, .$$
%Put $T = |E| \sqrt{2 \prob(\rho \in S)^{-1}}$ to conclude the proof by Lemma \ref{rwhitmiss}.
\end{proof}

\subsection{The star-tree transform}

Let $G$ be a graph. We define the {\em star-tree} transform $G^\dagger$ of $G$ as the graph of maximal degree at most $3$ obtained by the following operations (see Figure \ref{fig2.sttransform}).
\begin{enumerate}
\item We subdivide each edge $e$ of $G$ by adding a new vertex $w_e$ of degree $2$. Denote the resulting intermediate graph by $G'$.

\item Replace each vertex $v$ of $G$ and its incident edges in $G'$ by a balanced binary tree $T_v$ with $\deg(v)$ leaves which we identify with $v$'s neighbors in~$G'$. When $G$ is planar we choose this identification so as to preserve planarity, otherwise, this is an arbitrary identification. We denote by $w_v$ the root of $T_v$. Denote the resulting graph by $G^*$.
%Denote the resulting graph by $G^\da$.

\end{enumerate}

%\begin{remark} When $G$ is planar we choose the identification of the leaves of $T_v$ with the neighbors of $v$ in $G'$ such that $G^\da$ is planar.
%\end{remark}

\begin{figure}
\begin{center}
\includegraphics[width=0.95\textwidth]{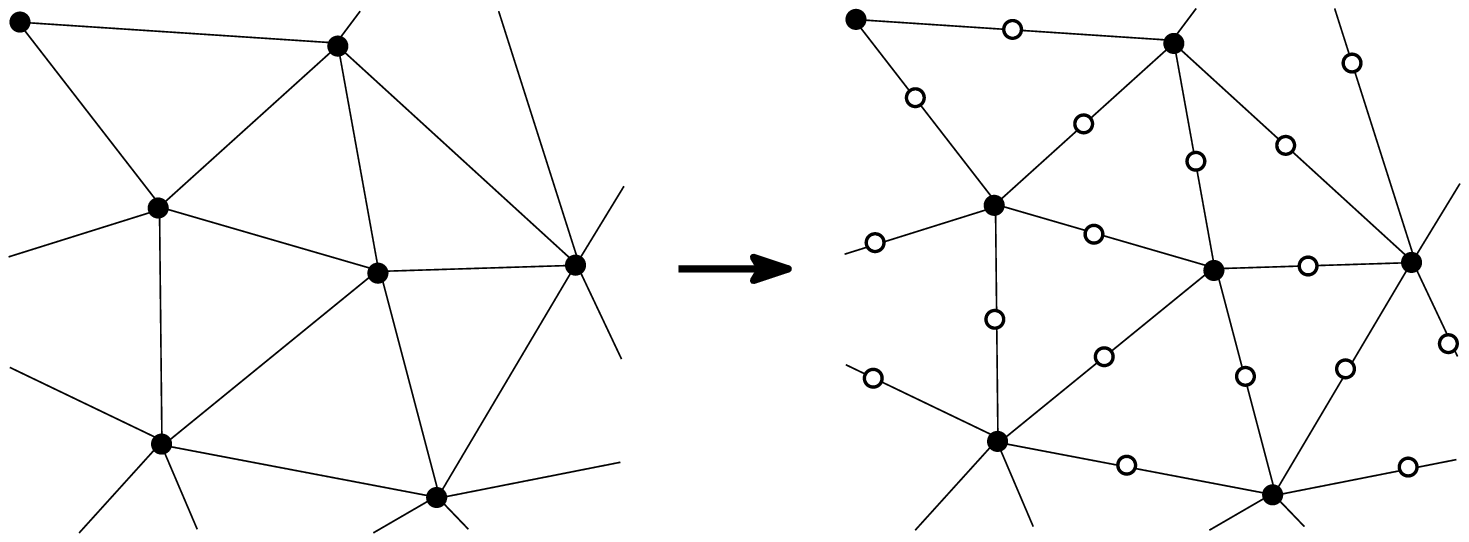}
%\caption{Stage 1 of the star-tree transform: subdividing edges}
%\label{fig1.sttransform}
\end{center}
\end{figure}

\begin{figure}
\begin{center}
\includegraphics[width=0.6\textwidth]{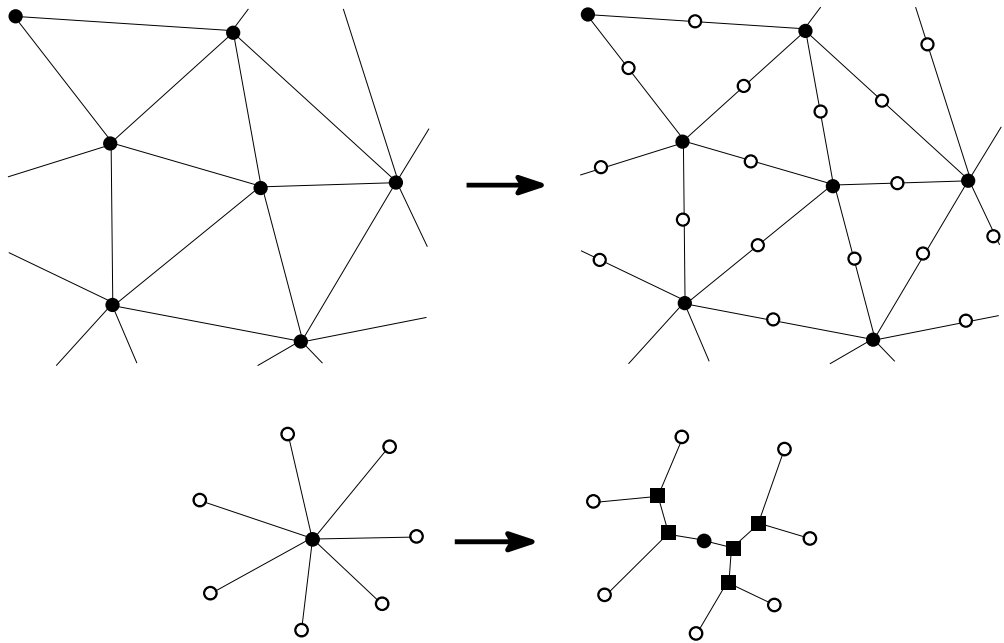}
\caption{The star-tree transform. Stage 1 (top): subdividing edges. Stage 2 (bottom): replacing stars with binary trees.}
\label{fig2.sttransform}
\end{center}
\end{figure}

%This turns each face (which was a triangle) into a $6$ sided face. In each of these connect the $3$ ``new'' vertices to divide the face into $4$ triangular faces. This results in a triangulation in which the original vertices retain the same degree and are separated from each other by new vertices, all of which have degree $6$.
%
%We place this binary tree inside the face with $\deg(v)$ edges that we got from erasing $v$.
%
%To each face that is not a triangle add edges to make it a triangle in such a way that to each vertex of the face we add at most $2$ edge incident to it. The resulting graph has degree bounded by $14$.

\begin{lemma} \label{starrecurrence} Let $G$ be an infinite connected graph and let $G^\da$ be its star-tree transform and equip $G^\da$ with edge resistances $R$ as follows: on each edge $e$ of the binary tree $T_v$ we put $R_e = 1/\deg(v)$ where $\deg(v)$ is the degree of $v$ in $G$. Then if $(G^\da, R)$ is recurrent, then $G$ is recurrent.
\end{lemma}
\begin{proof}
Assume that $G$ is transient. Equivalently, that there is a unit flow $\theta$ on $G$, from some $a\in V(G)$ to infinity, with $\En(\theta) < \infty$. Given this flow we will construct a unit flow $\theta^\da$ on $G^\da$, from $w_a$ to infinity, such that $\En(\theta^\da) \le 4 \En(\theta)$, thus showing that $G^\da$ is also transient.

First we define a flow $\theta'$ from $a$ to infinity in $G'$ in the natural manner: for each edge $e=(x,y)$ of $G$ we set $\theta'(x,w_e)=\theta'(w_e,y)=\theta(x,y)$. Obviously $\En(\theta')=2\En(\theta)$. Next we introduce some notation. Let $v$ be some vertex of $G$. Put $k= \deg(v)$ and note that the height of $T_v$ is $h=\lceil \log_2 k \rceil$ and that some leaves may occur at height $h-1$. Recall that at stage $2$ of the transform we obtain a correspondence between the $k$ leaves of the tree and the neighbors of $v$ in $G'$, denote the latter vertices by $v_1, \ldots, v_k$ and let $e_1, \ldots, e_k$ be the unique incident edges in $T_v$, respectively. Associate with each edge $e \in T_v$ of the tree a string $M(e) \in \{0,1\}^{\leq h}$ of $0$'s and $1$'s of length at most $h$. The string $M(e)$ ``codes'' the location of the edge in $T_v$ by recording left turns with $0$ and right turns with $1$ so that the height of $e$ is $|M(e)|$ (edges touching the root $w_v$ have height $1$).

We now construct the flow $\theta^\da$. For each edge $e=(x,y)$ of $T_v$ assume that $(x,y)$ points towards the root $w_v$ (recall that $\theta^\da$ should be antisymmetric) and set
\be\label{newflow} \theta^\da(e) =  \sum_{j : M(e) \preceq M(e_j)} \theta'(v_j,v) \, ,\ee
where two strings $S_1$ and $S_2$ satisfy $S_1 \preceq S_2$ if $S_1$ is a prefix of $S_2$. Let us first verify that this is a unit flow from $a$ to $\infty$. Indeed, let $u$ be a vertex in the tree that is not a leaf or the root and denote its two children by $u_1, u_2$ and its father by $u^+$. By our construction we have $\theta^\da(u_1,u) + \theta^\da(u_2,u) = \theta^\da(u, u+)$. If $u=w_v$ and $v \ne a$, then $\theta^\da(u_1,u)+\theta^\da(u_2,u)=0$ since $\theta'$ was a flow. If $u=w_a$, then $\theta^\da(u_1,u)+\theta^\da(u_2,u)=-1$. Lastly, when $u$ is a leaf of $T_v$ the corresponding vertex $v_j$ has degree $2$ and the flow passing through it is precisely the same as in $\theta'$.

Next we bound the energy $\En(\theta^\da)$ in terms of $\En(\theta')$. By (\ref{newflow}) and Cauchy-Schwarz inequality, for any edge $e$ of $T_v$ at height $\ell$ the contribution to $\En(\theta^\da)$ from $e$ is
$$ R_e\big [\theta^\da(e) \big ]^2 =  {1 \over k} \big [ \sum_{j : M(e) \preceq M(e_j)} \theta'(v_j,v) \big ]^2 \leq {2^{h-\ell} \over k} \sum_{j : M(e) \preceq M(e_j)} [\theta'(v_j,v)]^2 \, .$$
When summing the right hand side over all edges of $T_v$ the term $[\theta'(v_j,v)]^2$ appears once for each level $\ell = 1,\ldots, h$ with coefficient $k^{-1} 2^{h-\ell}$. Hence, the total contribution of the edges of $T_v$ to $\En(\theta^\da)$ is at most
$$ \sum_{\ell=1}^h \sum_{j=1}^k {2^{h- \ell} \over k} [\theta'(v_j,v)]^2 \leq 2\sum_{j=1}^k [\theta'(v_j,v)]^2  \, .$$
Thus, when summing over all $v$ we obtain that
$$ \En(\theta^*) \leq 2\En(\theta')  = 4 \En(\theta) \, ,$$
concluding the proof.
\end{proof}

\section{Proof of main results}

We begin by fixing some notation. Recall that we are given finite graphs $G_n$ and that $\rho_n$ is a randomly chosen vertex drawn from the stationary measure on $G_n$ and that $(U,\rho)$ is the distributional limit of this sequence. We write $(G_n^\dagger, \rho_n^\dagger)$ and $(U^\dagger, \rho^\dagger)$ for the result of the star-tree transform on $G_n$ and $U$ with the roots $\rho_n^\dagger$ and $\rho^\dagger$ chosen to be uniform vertices of the trees $T_{\rho_n}$ (in $G_n^\dagger$) and $T_\rho$ (in $U^\dagger$). We also set markings on $G_n^\dagger$ and $U^\da$ by putting $M(e)=\deg(v)$ for any edge $e$ in the tree $T_v$, where $\deg(v)$ is the degree of $v$ in $G_n$ or $U$, respectively.

\begin{lemma} \label{rootstarexpdecay} The triplet $(U^\da, \rho^\da, M)$ has an exponential tail.
\end{lemma}
\begin{proof} Note that $M(\rho^\da)=\deg(v)$ where $v$ is either $\rho$ or one of its neighbors. Hence it suffices to show that if $(U,\rho)$ is a distributional limit in which $\deg(\rho)$ has an exponential tail, then $D(\rho)=\max_{u: (u,\rho)\in E} \deg(u)$ also has an exponential tail. Indeed,
$$ \prob(D(\rho) \geq k) \leq \prob(\deg(\rho) \geq k) + \prob(\deg(\rho)\leq k \and D(\rho) \geq k) \, .$$
The probability of the first term on the right hand side decays exponentially. For the second term we have
$$ \prob \big ( \deg(X_1) \geq k \, \big | \, \deg(\rho)\leq k \and D(\rho) \geq k \big ) \geq k^{-1} \, ,$$
where $X_1$ is a random uniform neighbor of $\rho$. By stationarity $\prob(\deg(X_1) \geq k)$ decays exponentially, concluding the proof.
%Since the degree of $\rho$ in $(U,\rho)$ decays exponentially, the triplet $(U^\da, \rho^\da,M)$ has an exponential tail.
\end{proof}

We now provide a proof for the intuitive fact that $(U^\dagger, \rho^\dagger)$ is the distributional limit of $(G_n^\dagger, \rho_n^\dagger)$, see Figure \ref{fig.diagram}.
\begin{lemma} \label{startreects} The star-tree transform is continuous on the space of distributions on rooted graphs.
\end{lemma}
\begin{proof} Let $(H^*,h^*)$ be the star-tree transform of $(H,h)$ as defined above. Then for any fixed $m>0$ the distribution of $B_{H^*}(h^*,m)$ is determined by the distribution of $B_{H}(h,m)$ since the star-tree transform only increases distances.
\end{proof}

\begin{corollary} \label{commutative} $(U^\da,\rho^\da)$ is the distributional limit of $(G_n^\da, \rho_n^\da)$.
\end{corollary}

Note that $(G^\da_n,\rho^\da_n)$ and $(U^\da, \rho^\da)$ are {\em not} stationary with respect to the simple random walk. To overcome this small technicality, let $\rho_n^\pi$ be a random root chosen from the stationary distribution on $G_n^\da$ and write $(U^*, \rho^\pi)$ for an arbitrary subsequential distributional limit. Note that both $(G_n^\da, \rho_n^\pi,M)$ and $(U^*, \rho^\pi, M)$ are stationary.

%As in Section \ref{sec-dlimit}, $(G^\da_n,\rho^\da_n)$ is asymptotically contiguous to $(G^\da_n, \rho_n^\pi)$ hence $(U^\da, \rho^\da)$ is absolutely continuous with respect to $(U^\da, \rho^\pi)$.

\begin{lemma}\label{contiguous} There exists a universal constant $C>0$ such that for any $n$
$$ C^{-1} \leq {\prob_{(G_n^\da, \rho^\da_n)}(A) \over \prob_{(G_n^\da, \rho^\pi_n)}(A)} \leq C \, ,$$
for any event $A$ on random rooted graphs.
\end{lemma}
\begin{proof} Since $G_n^\da$ has bounded degree, the probability that $\rho_n^\pi=v$ for any vertex $v$ is, up to a multiplicative constant, ${1 \over |G_n^\da|}$. The same holds for $(G_n^\da, \rho^\da_n)$ because $\rho^\da_n$ was chosen uniformly from $T_{\rho_n}$ which has size proportional to $\deg(\rho_n)$ and $\rho_n$ was chosen with probability proportional to $\deg(\rho_n)$. Hence, for any fixed $n$ and an event $A$ of random rooted graphs we have that $P_{(G_n^\da, \rho^\da_n)}(A)$ and $P_{(G_n^\da, \rho^\pi_n)}(A)$ are the same up to a multiplicative constant.
%The absolute continuity follows immediately from the asymptotic contiguity and Corollary \ref{commutative}.
\end{proof}

\begin{corollary} \label{abscts} There exists a universal constant $C>0$ such that for any event $A$
$$ C^{-1} \leq {\prob_{(U^\da, \rho^\da)}(A) \over \prob_{(U^\da, \rho^\pi)}(A)} \leq C \, .$$
In particular, $(U^\da,\rho^\da)$ and $(U^\da,\rho^\pi)$ are absolutely continuous with respect to each other.
\end{corollary}
\begin{proof} Immediate from Lemma \ref{contiguous} and Corollary \ref{commutative}.
\end{proof}

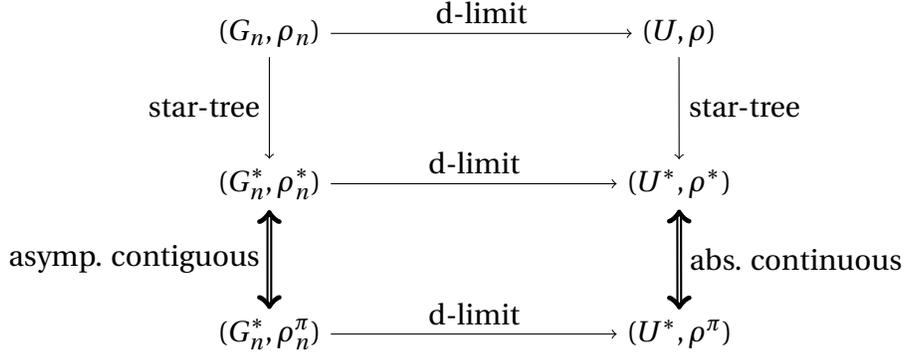
\begin{figure}
\centering
\begin{tikzpicture}[node distance=2cm, auto]\label{fid.diagram}
  \node (Gn) {$(G_n, \rho_n)$};
  \node (Urho) [right=4cm of Gn] {$(U,\rho)$};
  \node (Gnstar) [below of=Gn] {$(G_n^\dagger, \rho_n^\dagger)$};
  \node (Urhostar) [right=3.8cm of Gnstar] {$(U^\dagger, \rho^\dagger)$};
  \node (Gnstarst) [below of=Gnstar] {$(G_n^\dagger, \rho_n^{\pi})$};
  \node (Urhostarst) [right=3.8cm of Gnstarst] {$(U^\dagger, \rho^{\pi})$};
  \draw[->] (Gn) to node {d-limit} (Urho);
  \draw[->] (Gn) to node [swap] {star-tree} (Gnstar);
  \draw[->] (Gnstar) to node {d-limit} (Urhostar);
  \draw[->] (Urho) to node {star-tree} (Urhostar);
  \draw[<->,double, shorten <=.5pt, thick] (Gnstar) to node [swap] {asymp. contiguous} (Gnstarst);
  \draw[->] (Gnstarst) to node {d-limit} (Urhostarst);
  \draw[<->,double, shorten <=.5pt, thick] (Urhostar) to node {abs. continuous} (Urhostarst);
\end{tikzpicture}
\caption{Commutative diagram} \label{fig.diagram}
\end{figure}

\noindent {\bf Proof of Theorem \ref{mainthm}.} For convenience, we use below the standard $O$-notation. Given two sequences of non-negative numbers $f(k), g(k)$ we write $f = O(g)$ or $g = \Omega(f)$ if there exists $C<\infty$ such that $f(k) \leq Cg(k)$ for sufficiently large $k$. We write $f = \Theta(g)$ if $f = O(g)$ and $g=O(f)$.

By definition, $(U^\da,\rho^\pi,M)$ is a distributional limit of bounded degree finite planar graphs and by Lemma \ref{rootstarexpdecay} and Corollary \ref{abscts}, it has an exponential tail. Apply Theorem \ref{quantrecurrence} to get that almost surely there exist subsets $B_k \subset U^\da$ with $|B_k| = O(k)$ such that
\be\label{reffestimate} \Reff(\rho^\pi \lra U^\da \setminus B_k) \geq \Omega(\log k) \, .\ee
Assume without loss of generality that $|B_k| = \Theta(k)$ (otherwise we can always add vertices and the resistance above only grows). We define three different networks with underlying graph $U^\dagger$ by specifying the edge resistances for each edge $e$ as follows,
\begin{eqnarray*}
R^{{\rm{unit}}}_e = 1 \qquad \and \qquad R^{{\rm{mark}}}_e = M(e)^{-1} \, ,
\end{eqnarray*}
and given some $s>0$ we define
$$ R^{s}_e = \begin{cases} 1 &\mbox{if } M(e) \leq s \, ,\\
M(e)^{-1} &\mbox{otherwise} \, .\end{cases}$$
Now, since $|B_k| = \Theta(k)$ and (\ref{reffestimate}), Lemma \ref{highmarksdontmatter} together with (\ref{rwresist}) gives
$$ \Reff \big (\rho^{\pi} \lra U^\da \setminus B_k \, ; \, \{R^{{\rm{C\log k}}}_e\}  \big ) \geq \Omega(\log k) \, ,$$
for some $C<\infty$ depending only on the exponential decay rate. For any $m \geq 0$, by Lemma \ref{resistintriangle} with $A=B_{U^\da}(\rho^{\pi},m)$ we have
$$ \Reff \big (B_{U^\da}(\rho^{\pi},m) \lra U^\da \setminus B_k \, ; \, \{R^{{\rm{C\log k}}}_e\}  \big ) \geq \Omega(\log k) - m \, ,$$
since $R^{C \log k}_e \leq 1$ for all $e$ hence $\Reff(\rho^\pi \lra v) \leq m$ for all $v\in B_{U^\da}(\rho^{\pi},m)$. We have $R^{{\rm{mark}}}_e \geq [C \log k]^{-1} R^{C \log k}_e$ for all $e$, hence
$$ \Reff \big (\partial B_{U^\da}(\rho^{\pi},m) \lra U^* \setminus B_k \, ; \, \{R^{{\rm mark}}_e\} \big ) \geq \Omega(1) - O(m / \log k) \, .$$
All this occurs almost surely in $(U^\da, \rho^\pi)$. Corollary \ref{abscts} shows that almost surely
$$ \Reff \big (\partial B_{U^\da}(\rho^{\da},m) \lra U^* \setminus B_k \, ; \, \{R^{{\rm mark}}_e\} \big ) \geq \Omega(1) - O(m / \log k) \, ,$$
for any fixed $m \geq 0$. By (\ref{recurrencecriterion}) we deduce that the network $(U^\da, \rho^\da)$ with edge resistances $\{R^{{\rm mark}}_e\}$ is almost surely recurrent. Lemma \ref{starrecurrence} implies that $U$ is almost surely recurrent, concluding our proof. \qed \\

%Using this result we are already ready to prove Theorem \ref{mainthm3}. \\

\noindent {\bf Proof of Corollaries \ref{maincor} and \ref{maincor2}.} Follows immediately since the UIPT and UIPQ are distributional limits of finite planar graphs, and it is known that the degree of the root has an exponential tail, see \cite[Lemma 4.1, 4.2]{AS} and \cite{GR} for the UIPT and \cite[Proposition 9]{BC1} for the UIPQ. \qed \\

\section*{Acknowledgements} We are indebted to Omer Angel for many fruitful discussions and for comments on a previous version of this manuscript. We are grateful to Itai Benjamini, Tom Meyerovitch, Gr\'egory Miermont, Gourab Ray, Juan Souto for many useful conversations. We also thank the participants of the circle packing reading seminar held in UBC at the fall of 2011.
We thank Nicolas Curien for bringing \cite[Proposition 9]{BC1} to our attention.
%The exploration analysis in the proof of Lemma \ref{rootdegree} was suggested to us by Omer Angel and Gr\'egory Miermont.
%\begin{thebibliography}{99}

%\end{thebibliography}
\end{document}